\documentclass[11pt]{amsart}

\usepackage{amsmath,amssymb,amscd,url}
\usepackage{xcolor}

\usepackage{amsmath,amssymb}

\usepackage{amsxtra, amsmath}
\usepackage{amssymb, amscd}
\usepackage{graphicx}
\usepackage{mathrsfs}

\usepackage{mathtools}

\newtheorem{thm}{Theorem}[section]
\newtheorem{prop}[thm]{Proposition}

\theoremstyle{definition}
\newtheorem{definition}[thm]{Definition}

\theoremstyle{remark}
\newtheorem{remark}[thm]{Remark}
\theoremstyle{definition}
\newtheorem{example}[thm]{Example}

\usepackage{dsfont}
\allowdisplaybreaks

\title{Explicit construction of matrix-valued orthogonal polynomials of arbitrary size}
\author{Ignacio Bono Parisi}

\address{CIEM-FaMAF, Universidad Nacional de C\'ordoba, C\'ordoba~5000, Argentina}
\email{ignacio.bono@unc.edu.ar}

\thanks{This paper was partially supported by SeCyT-UNC, CONICET, PIP 11220200102031CO}

\begin{document}
\begin{abstract}
    In this paper, we explicitly provide expressions for a sequence of orthogonal polynomials associated with a weight matrix of size $N$ constructed from a collection of scalar weights $w_{1}, \ldots, w_{N}$:
    $$W(x) = T(x)\operatorname{diag}(w_{1}(x), \ldots, w_{N}(x))T(x)^{\ast},$$  
    where $T(x)$ is a specific polynomial matrix. We provide sufficient conditions on the scalar weights to ensure that the weight matrix $W$ is irreducible. Furthermore, we give sufficient conditions on the scalar weights to ensure that each term in the constructed sequence of matrix orthogonal polynomials is an eigenfunction of a differential operator. We also study the Darboux transformations and bispectrality of the orthogonal polynomials in the particular case where the scalar weights are the classical weights of Jacobi, Hermite, and Laguerre.   
    With these results, we construct a wide variety of bispectral matrix-valued orthogonal polynomials of arbitrary size, which satisfy a second-order differential equation. 
\end{abstract} 
\maketitle

\section{Introduction}
The theory of matrix-valued orthogonal polynomials, which are matrix polynomials orthogonal with respect to a matrix-valued inner product defined by a weight matrix $W(x)$, was initiated in 1949 by M. G. Krein \cite{K49} and \cite{K71}. Since their origin, matrix-valued orthogonal polynomials have extended the classical theory of scalar polynomials to a richer matrix context. This generalization has provided new tools and insights in various mathematical fields, including spectral theory, integrable systems, operator theory, and special functions.

These polynomials are of particular interest when they are eigenfunctions of a matrix-valued differential operator, extending the classical scalar orthogonal polynomials of Hermite, Laguerre, and Jacobi. 

The classification of all weight matrices $W(x)$ whose associated sequences of orthogonal polynomials are eigenfunctions of a second-order differential operator is known as the Matrix Bochner Problem. This problem was initially posed by S. Bochner \cite{B29} for the scalar case and later extended to the matrix case by A. Durán \cite{D97}. While this problem is solved in the scalar case (S. Bochner \cite{B29}), it remains open in the matrix setting. A major breakthrough was achieved by R. Casper and M. Yakimov \cite{CY18}, who provided a full classification of solutions under a natural additional hypothesis. Complementary progress was made in \cite{BP23} and \cite{BP24-1}, where the authors constructed irreducible weight matrices that do not satisfy this hypothesis, showing the existence of solutions beyond the initial classification established by R. Casper and M. Yakimov.

In general, given a weight matrix $W(x)$, obtaining explicit and manageable expressions for the associated orthogonal polynomials is a challenging task. There are some techniques in the literature to obtain this expression like the Rodrigues formula, which involves computing the $n$-th derivatives of the weight matrix, to derive expressions for matrix-valued orthogonal polynomials (see \cite{IKP19}, \cite{DG07}, \cite{DL07}, \cite{KR19}, \cite{BCD12}, \cite{CC22}, \cite{DD08}).  Another approach to obtaining an expression for a sequence of orthogonal polynomials is by using the Darboux transformation, which involves factorizing an operator and then mapping a known sequence of orthogonal polynomials to a new sequence via a differential operator (see \cite{C18}, \cite{BPstrong}, \cite{BPZ24}, \cite{CY18}). However, these techniques require the weight matrix involved to be sufficiently well-behaved, and even then, they do not always yield manageable formulas. 

Explicit expressions of matrix-valued orthogonal polynomials are crucial, as they provide concrete examples for verifying and extending well-known properties from scalar orthogonal polynomials to the matrix setting. Furthermore, these expressions are indispensable for applications, such as the time and band limiting problem over a noncommutative ring and matrix-commutative operators (see \cite{CGYZ24,GPZ20,GPZ17,GPZ15,CG17,CG15}).

\

In this paper, we construct explicit expressions for matrix-valued orthogonal polynomials associated with specific weight matrices of the form 
$$W(x) = T(x) \operatorname{diag}(w_{1}(x),\ldots,w_{N}(x))T(x)^{\ast},$$
where $T(x) = e^{Ax}$, with $A$ a nilpotent matrix and $w_{1},\ldots,w_{N}$ are scalar weights. For certain choices of the scalar weights $w_i$, the resulting weight matrices have already appeared in the literature. Specifically, for $N = 2$, they were studied in \cite{DG04} for the Hermite and Laguerre scalar weights, and for arbitrary size, they appear in \cite{BP24-1} when the scalar weights $w_i$ are classical weights of the same type and satisfy the condition that $w_i(x)/w_j(x)$ is not a rational function. The latter cases correspond to solutions of the Bochner problem that fall outside the scope of the classification theorem by Casper and Yakimov \cite{CY18}.

In this work, we extend the freedom in the choice of the scalar weights $w_i$, allowing them to go beyond the classical examples. 
Our main result is Theorem \ref{main}, in which leveraging the structure of these weight matrices, we derive simple and explicit expression for a sequence $Q_{n}(x)$ of orthogonal polynomials for $W(x)$ in terms of the monic orthogonal polynomials $p_{n}^{w_{i}}(x)$ of the scalar weights $w_{i}(x)$. The expression we obtain is quite manageable and is given by  
$$Q_{n}(x) = \begin{psmallmatrix}p_{n}^{w_{1}}(x) & a_{1}p_{n+1}^{w_{2}}(x) & 0 & 0 & 0 & \cdots \\ 
-a_{1} \frac{\|p_{n}^{w_{2}}\|^{2}}{\|p_{n-1}^{w_{1}}\|^{2}} p_{n-1}^{w_{1}}(x) & p_{n}^{w_{2}}(x) & -a_{2} \frac{\|p_{n}^{w_{2}}\|^{2}}{\|p_{n-1}^{w_{3}}\|^{2}}p_{n-1}^{w_{3}}(x) & 0 & 0 & \cdots \\ 0 &  a_{2} p_{n+1}^{w_{2}}(x) & p_{n}^{w_{3}}(x) & a_{3}p_{n+1}^{w_{4}}(x) & 0 & \cdots \\ 
0 & 0 & -a_{3} \frac{\|p_{n}^{w_{4}}\|^{2}}{\|p_{n-1}^{w_{3}}\|^{2}}p_{n-1}^{w_{3}}(x) & p_{n}^{w_{4}}(x) & -a_{4} \frac{\|p_{n}^{w_{4}}\|^{2}}{\|p_{n-1}^{w_{5}}\|^{2}}p_{n-1}^{w_{5}}(x) & \cdots \\
0 & 0 & 0 & a_{4} p_{n+1}^{w_{4}}(x) & p_{n}^{w_{5}}(x) & \cdots \\
\vdots & \vdots & \vdots & \vdots & \vdots & \ddots\end{psmallmatrix}T(x)^{-1}.$$ 
In this way, we can create an extension of scalar orthogonal polynomials to matrix-valued orthogonal polynomials and generate a wide variety of examples.

For the benefit of the reader, in \eqref{W 2} and \eqref{W 3} we present the explicit expressions of the weight $W$ and the polynomials $Q_{n}(x)$  for the particular cases when $N = 2$ and $N = 3$.

A natural question that arises is whether these matrix-valued polynomials are bispectral. In Theorem \ref{bisp}, we establish a sufficient condition on the scalar orthogonal polynomials \( p_{n}^{w_{i}} \) to ensure that the matrix polynomials constructed in Theorem \ref{main} are eigenfunctions of a differential operator. Using this result, in Section \ref{S4} we construct several families of bispectral sequences of matrix-valued orthogonal polynomials from the classical Hermite, Laguerre, and Jacobi polynomials. We even construct a novel sequence of bispectral polynomials by combining weights from different classical families, specifically a Hermite weight and a Laguerre weight. To the best of our knowledge, this approach of mixing weights from distinct families has not been explored in the literature before.

We take advantage of the explicit expressions of these orthogonal polynomials to analyze when they can be connected to direct sums of classical scalar orthogonal polynomials via Darboux transformations. Furthermore, these expressions offer a straightforward approach to deriving the three-term recurrence relation and exploring additional properties of the polynomials.

Finally, we study the irreducibility of the weight matrices $W$, which depends on the choice of the scalar weights $w_i$. In Section \ref{S5}, we show that there are many situations in which $W$ is an irreducible weight. Thus, the choice of scalar weights $w_i$ is not as restricted as it might initially appear, and a wide range of irreducible examples can be constructed.

The paper is organized as follows. In section \ref{S2} we give the preliminaries of matrix-valued orthogonal polynomials, the Darboux transformation, and the classical scalar polynomials. In Section \ref{S3}, we introduce the weight matrix $W$ constructed from a finite collection of scalar weights supported on the same interval. In Theorem \ref{main}, we provide an explicit expression for a sequence of matrix-valued orthogonal polynomials for $W$ in terms of the scalar monic orthogonal polynomials associated with the collection of scalar weights.
In Section \ref{S4}, we study explicit examples constructed from the classical weights in order to obtain examples that are bispectral, and we study the relationship of these examples with Darboux transformations of a direct sum of scalar weights. 
Finally, in Section \ref{S5}, we discuss the irreducibility of the weights $W$ constructed in Section \ref{S3}.

\section{Background}\label{S2}
\subsection{Orthogonal polynomials and the algebra $\mathcal{D}(W)$} 

\begin{definition}
    A weight matrix $W$ of size $N$ supported on an interval $(x_{0},x_{1})$, is an integrable function $W:\mathbb{R}\to\operatorname{Mat}_{N}(\mathbb{C})$ such that $W(x)$ is definite positive almost everywhere in $(x_{0},x_{1})$, $W(x) = 0$ if $x \notin (x_{0},x_{1})$, and $\int_{x_{0}}^{x_{1}}x^{n}W(x)dx < \infty$ for all $n\geq 0$. 
\end{definition}

For a weight matrix $W$ one can induce an inner product in $\operatorname{Mat}_{N}(\mathbb{C}[x])$ given by 
\begin{equation}\label{inner prod}
    \langle P , Q \rangle  = \int_{x_{0}}^{x_{1}}P(x)W(x)Q(x)^{\ast}dx, \text{ for all } P,Q \in \operatorname{Mat}_{N}(\mathbb{C}[x]).
\end{equation}
By using the above inner product, one can construct a sequence of orthogonal polynomials $Q_{n}(x)$ for the weight matrix $W$. That is, for each $n\in \mathbb{N}_{0}$, $Q_{n}(x)$ is a polynomial of degree $n$ with leading coefficient a nonsingular matrix, and $\langle Q_{n}, Q_{m} \rangle = 0$ for all $n \not = m$. Moreover, if $R_{n}(x)$ is another sequence of matrix-valued orthogonal polynomials for $W$, then $R_{n}(x) = M_{n}Q_{n}(x)$ for some sequence of invertible matrices $M_{n} \in \operatorname{Mat}_{N}(\mathbb{C})$.  By a standard argument (see \cite{K49} or \cite{K71}), one can prove that the sequence $Q_{n}$ satisfies a three-term recurrence relation of the form 
\begin{equation} \label{three-term}
    Q_{n}(x)x = A_{n}Q_{n+1}(x) + B_{n}Q_{n}(x) + C_{n}Q_{n-1}(x)
\end{equation}
for some sequence of constant matrices $A_{n},B_{n},C_{n} \in \operatorname{Mat}_{N}(\mathbb{C})$.

In other words, the sequence $Q_{n}(x)$ is an eigenfunction of a discrete operator $\mathscr{L} = A_{n}\delta + B_{n} + C_{n}\delta^{-1}$, where $\delta^{j}$ acts on the left-hand side on a sequence as $\delta^{j} \cdot p_{n} = p_{n+j}$, for $j \in \mathbb{Z}$. We have that 
\begin{equation}\label{discL}
    \mathscr{L} \cdot Q_{n}(x) = Q_{n}(x)x.
\end{equation}

\

Throughout this paper, we consider differential operators
$$D = \sum_{j=0}^{m}\partial^{j}F_{j}(x),$$
where $\partial = \frac{d}{dx}$, and $F_{j}(x)$ is a matrix-valued function.
These operators act on the right-hand side of matrix-valued functions as follows 
$$P(x)\cdot D = \sum_{j=0}^{m}\partial^{j}(P)(x)F_{j}(x).$$
We introduce the algebra $\mathcal{D}(W)$.
\begin{definition}
    Given a weight matrix $W$, we introduce the algebra $\mathcal{D}(W)$ of all differential operators that have a sequence of orthogonal polynomials for $W$ as eigenfunctions. That is,
    $$\mathcal{D}(W) = \left \{ D = \sum_{j=0}^{m}\partial^{j}F_{j} \, : \, Q_{n}(x) \cdot D = \Lambda_{n}(D)Q_{n}(x), \text{ with } \Lambda_{n}(D) \in \operatorname{Mat}_{N}(\mathbb{C}), \text{ for all } n \geq 0 \right \},$$
    where $Q_{n}(x)$ is a sequence of orthogonal polynomials for $W$. 
\end{definition}
\begin{remark}
    The definition of $\mathcal{D}(W)$ is independent of the choice of the sequence of orthogonal polynomials for $W$. This is because any two sequences of orthogonal polynomials, $Q_{n}$ and $R_{n}$, for $W$ differ by an invertible matrix $M_{n}$. Consequently, $Q_{n}(x) \cdot D = \Lambda_{n}(D) Q_{n}(x)$ if and only if $R_{n}(x) \cdot D = M_{n} \Lambda_{n}(D) M_{n}^{-1} R_{n}(x)$. Thus, a sequence of orthogonal polynomials $Q_{n}$ is an eigenfunction of an operator $D$ if and only if any other sequence $R_{n}$ for the same weight $W$ is also an eigenfunction of $D$.
\end{remark}

If there exists a nonconstant differential operator $D \in \mathcal{D}(W)$, then, along with the operator defined in \eqref{discL}, $Q_{n}$ is simultaneously an eigenfunction of both a discrete operator and a differential operator:
$$\mathscr{L} \cdot Q_{n}(x) = Q_{n}(x) x, \quad \text{and} \quad Q_{n}(x) \cdot D = \Lambda_{n}(D) Q_{n}(x).$$
In this case, we say that $Q_{n}(x)$ is bispectral.

\

Given two weight matrices $W$ and $\tilde{W}$ supported on $(x_{0},x_{1})$, we say that $W$ is \emph{similar} to $\tilde{W}$ if there exists a nonsingular matrix $M \in \operatorname{Mat}_{N}(\mathbb{C})$ such that $\tilde{W}(x) = MW(x)M^{\ast}$, for all $x \in (x_{o},x_{1})$. In this case, we have that $Q_{n}(x)$ is a sequence of orthogonal polynomials for $W$ if and only if $\tilde{Q}_{n}(x) = M{Q}_{n}(x)M^{-1}$ is a sequence of orthogonal polynomials for $\tilde{W}$, and their associated algebras satisfy that $\mathcal{D}(\tilde{W}) = M\mathcal{D}(W)M^{-1}$.

\begin{definition} A weight matrix $W$ of size $N$ supported on $(x_{0},x_{1})$ is said to be \emph{reducible} if there exist weight matrices $W_{1}$ and $W_{2}$ of sizes $N_{1}$ and $N_{2}$, respectively, with $N = N_{1} + N_{2}$, and a nonsingular matrix $M \in \operatorname{Mat}_{N}(\mathbb{C})$ such that 

$$W(x) = M\begin{psmallmatrix} W_{1}(x) & 0 \\ 0 & W_{2}(x) \end{psmallmatrix} M^{\ast}, \quad \text{ for all } x \in (x_{0},x_{1}).$$
We say that $W$ is an \emph{irreducible weight} if it is not reducible. \end{definition}

\subsection{The Darboux Transformation}
In this subsection, we introduce the notion of the Darboux transformation between weight matrices. 
\begin{definition}
    We say that a differential operator $D$ is a \emph{degree-preserving} operator if, for every polynomial $P$ of degree $n$, it follows that $P(x)\cdot D$ is a polynomial of degree $n$ for all but finitely many $n$.
\end{definition}
\begin{definition} \label{darb def}
    Let $W$ and $\tilde{W}$ be weight matrices of size $N$, and let $Q_{n}(x)$ and $\tilde{Q}_{n}(x)$ be a sequence of orthogonal polynomials for $W$ and $\tilde{W}$, respectively. We say that $\tilde{W}$ is a Darboux transformation of $W$ if there exists a differential operator $D \in \mathcal{D}(W)$ that can be factored as $D = \mathcal{D}_{1}\mathcal{D}_{2}$, with $\mathcal{D}_{1}$ and $\mathcal{D}_{2}$ degree-preserving operators, such that
    $$Q_{n}(x) \cdot \mathcal{D}_{1} = A_{n} \tilde{Q}_{n}(x),$$
    for some sequence of constant matrices $A_{n}$, for all $n \geq 0$.
\end{definition}
As a consequence of the above definition, it follows that the operator $\tilde{D} = \mathcal{D}_{2}\mathcal{D}_{1}$ belongs to the algebra $\mathcal{D}(\tilde{W})$.

\

A sufficient condition was given in \cite{BPZ24} to ensure that $\tilde{W}$ is a Darboux transformation of $W$ without requiring the factorization of a differential operator.    
\begin{prop}[Theorem 3.3, \cite{BPZ24}]\label{darb eq}
    Let $W$ and $\tilde{W}$ be weight matrices, and let $Q_{n}$ and $\tilde{Q}_{n}$ sequences of orthogonal polynomials for $W$ and $\tilde{W}$, respectively. If there exists a differential operator $\mathcal{D}_{1}$ such that 
    $$P_{n}(x)\cdot \mathcal{D}_{1} = A_{n}\tilde{P}_{n}(x), \text{ for all } n\in\mathbb{N}_{0},$$
    with $A_{n} \in \operatorname{Mat}_{N}(\mathbb{C})$ nonsingular for all but finitely many $n\in \mathbb{N}_{0}$, then $\tilde{W}$ is a Darboux transformation of $W$.
\end{prop}

\subsection{Classical orthogonal polynomials} \label{cl}
We recall some properties of the classical scalar polynomials of Hermite, Laguerre, and Jacobi.

\subsubsection{Shifted Hermite Polynomials}
The monic shifted Hermite polynomials $h_{n}(x-b)$, $b \in \mathbb{R}$, are orthogonal with respect to the weight $w_{b}(x) = e^{-x^{2}+2bx}$ supported on $(-\infty, \infty)$. They have squared norm $\|h_{n}(x-b)\|^{2} = \sqrt{\pi}e^{b^{2}}n!2^{-n}$ and can be expressed using the Rodrigues formula:
$$h_{n}(x) = \frac{(-1)^{n}}{2^{n}}e^{x^{2}}\frac{d^{n}}{dx^{n}}e^{-x^{2}}.$$
The sequence of orthogonal polynomials satisfies the second-order equation 
$$h_{n}(x-b)\cdot (\partial^{2} + \partial (-2(x-b))) = -2nh_{n}(x-b).$$

\subsubsection{Laguerre Polynomials}
The monic Laguerre polynomials $\ell^{(\alpha)}_{n}(x)$, $\alpha > -1$, are orthogonal with respect to the weight $w_{\alpha}(x) = e^{-x}x^{\alpha}$ supported on $(0, \infty)$. Their squared norm is $\|\ell_{n}^{(\alpha)}(x)\|^{2} = n!\Gamma(n+\alpha+1)$ and they can be expressed by the Rodrigues formula:
$$\ell^{(\alpha)}_{n}(x) = (-1)^{n}e^{x}x^{-\alpha}\frac{d^{n}}{dx^{n}}(e^{-x}x^{n+\alpha}).$$
The sequence of orthogonal polynomials satisfies the second-order equation 
$$\ell_{n}^{(\alpha)}(x)\cdot (\partial^{2}x + \partial (\alpha+1-x)) = -n\ell_{n}^{(\alpha)}(x).$$

\subsubsection{Jacobi Polynomials}
The monic Jacobi polynomials $J^{(\alpha,\beta)}_{n}(x)$, $\alpha,\beta > -1$, are orthogonal with respect to the weight $w_{\alpha,\beta}(x) = (1-x)^{\alpha}(1+x)^{\beta}$ supported on $(-1, 1)$. They have squared norm $\|J_{n}^{(\alpha,\beta)}(x)\|^{2} = 2^{2n+\alpha+\beta+1}\, n! \, \frac{\Gamma(n+\alpha+1)\Gamma(n+\beta+1)\Gamma(n+\alpha+\beta+1)}{\Gamma(2n+\alpha+\beta+2)\Gamma(2n+\alpha+\beta+1)}$ and are given by:
$$J^{(\alpha,\beta)}_{n}(x) = \frac{(-1)^{n}}{(n+\alpha+\beta+1)_{n}}(1-x)^{-\alpha}(1+x)^{-\beta}\frac{d^{n}}{dx^{n}}((1-x)^{\alpha+n}(1+x)^{n+\beta}).$$
The sequence of orthogonal polynomials satisfies the second-order equation 
$$J_{n}^{(\alpha,\beta)}(x)\cdot (\partial^{2}(1-x^{2}) + \partial (\beta-\alpha-x(\alpha+\beta+2))) = -n(n+\alpha+\beta+1)J_{n}^{(\alpha,\beta)}(x).$$

\section{Construction of Matrix-Valued Orthogonal Polynomials} \label{S3}
This section aims to present our main theorem. Given a collection of scalar weights supported on the same interval, we construct a weight matrix $W$, and we explicitly give a sequence of matrix-valued orthogonal polynomials for $W$ in terms of the sequences of orthogonal polynomials of the initial scalar weights.

\

Let $w_{j}(x)$ be a scalar weight supported on $(x_{0},x_{1})$, for $j = 1, \ldots, N$. We consider the $N\times N$ weight matrix given by 
\begin{equation}\label{inner1}
    \tilde{W}(x) = \operatorname{diag}(w_{1}(x),\ldots,w_{N}(x)).
\end{equation}

Let $A$ be the nilpotent matrix
\begin{equation}\label{A}
    A = \sum_{j=1}^{[N/2]} a_{2j-1}E_{2j-1,2j} + \sum_{j=1}^{[(N-1)/2]}a_{2j}E_{2j+1,2j}, \quad a_{j} \in \mathbb{R}-\{0 \},
\end{equation} 

$$A = \begin{pmatrix} 0 & a_{1} & 0 & 0 & 0 & 0 & \ldots \\
0 & 0 & 0 & 0 & 0 & 0 & \ldots \\
0 & a_{2} & 0 & a_{3} & 0 & 0 & \ldots \\
0 & 0 & 0 & 0 & 0 & 0 & \ldots \\ 
0 & 0 & 0 & a_{4} & 0 & a_{5} & \ldots \\
0 & 0 & 0 & 0 & 0 & 0 & \ldots \\
\vdots & \vdots & \vdots & \vdots & \vdots & \vdots & \ddots \end{pmatrix}.$$

We introduce the weight matrix $W(x)$ given by 
\begin{equation}\label{B}
    W(x) = T(x)\tilde{W}(x)T(x)^{\ast},
\end{equation}
where $T(x)$ is the exponential matrix $T(x) = e^{Ax} = Ax + I$.

We want to determine a sequence of orthogonal polynomials for $W$ in terms of the monic scalar orthogonal polynomials for $w_{i}$. Let $p_{n}^{w_{j}}(x)$ be the sequence of monic orthogonal polynomials for the scalar weight $w_{j}$. We have that $P_{n}(x) = \operatorname{diag}(p_{n}^{w_{1}}(x), \ldots , p_{n}^{w_{N}}(x))$ is a sequence of monic orthogonal polynomials for the diagonal weight matrix $\tilde{W}$ defined in \eqref{inner1}. The squared norm of $P_{n}(x)$ is given by 
$$\|P_{n}\|^{2} = \langle P_{n}, P_{n} \rangle_{\tilde{W}} = \int_{x_{0}}^{x_{1}} P_{n}(x)\tilde{W}(x)P_{n}(x)^{\ast}dx = \operatorname{diag}(\|p_{n}^{w_{1}}\|^{2},\ldots,\|p_{n}^{w_{N}}\|^{2}),$$
which is an invertible matrix for all $n \geq 0$. With these elements in place, we can now state the following theorem.

\begin{thm}\label{main}
    Let $w_{1}, \ldots, w_{N}$ be scalar weights supported on $(x_{0},x_{1})$. Let $A$ be the nilpotent matrix defined in $\eqref{A}$, $W(x) = T(x)\tilde{W}(x)T(x)^{\ast}$ be the weight matrix defined in \eqref{B}, with $\tilde{W}(x) = \operatorname{diag}(w_{1}(x),\ldots,w_{N}(x))$. Let $p_{n}^{w_{j}}(x)$ be the sequence of monic orthogonal polynomials for $w_{j}(x)$, and $P_{n}(x) = \operatorname{diag}(p_{n}^{w_{1}}(x),\ldots,p_{n}^{w_{N}}(x))$. The following definition introduces a sequence of matrix-valued orthogonal polynomials
    for $W$  
    \begin{equation}
        \begin{split}
            Q_{n}(x) & = P_{n}(x)+AP_{n+1}(x)-\|P_{n}\|^{2}A^{\ast}\|P_{n-1}\|^{-2}P_{n-1}(x)-P_{n}(x)Ax\\
            & \quad +\|P_{n}\|^{2}A^{\ast}\|P_{n-1}\|^{-2}P_{n-1}(x)Ax.
        \end{split}
    \end{equation}
For $n=0$, we take $\|P_{-1}\|^{-2}$ to be $0$, i.e. $Q_{0}(x) = P_{0}(x)+AP_{1}(x)-P_{0}(x)Ax.$
\end{thm}
\begin{proof}
 Let $\mathcal{M}$ be the set
 \begin{equation}\label{set M}
 \mathcal{M} = \left \{  M = \sum_{j=1}^{[N/2]} m_{2j-1}E_{2j-1,2j} + \sum_{j=1}^{[(N-1)/2]}m_{2j}E_{2j+1,2j}, \quad m_{j} \in \mathbb{C} \right \}.
 \end{equation}
We note that the nilpotent matrix $A$ belongs to $\mathcal{M}$. Given a diagonal matrix $D = \operatorname{diag}(d_{1}, \ldots, d_{N})$ and $M \in \mathcal{M}$, it follows that 
\begin{equation}\label{DM}
    \begin{split}
        DM & = \sum_{j=1}^{[N/2]} d_{2j-1}m_{2j-1}E_{2j-1,2j} + \sum_{j=1}^{[(N-1)/2]}d_{2j+1}m_{2j}E_{2j+1,2j} \in \mathcal{M}, \\
        MD & = \sum_{j=1}^{[N/2]} d_{2j}m_{2j-1}E_{2j-1,2j} + \sum_{j=1}^{[(N-1)/2]}d_{2j}m_{2j}E_{2j+1,2j} \in \mathcal{M}.
    \end{split}
\end{equation}
Besides, it is easy to check that 
\begin{equation}\label{MM}
    M_{1}M_{2} = 0 \quad \text{for all} \quad  M_{1}, M_{2} \in \mathcal{M}.
\end{equation}

By \eqref{DM} we obtain that 
\begin{equation*}
    \begin{split}
        AP_{n+1}(x) - P_{n}(x)Ax & = \sum_{j=1}^{[N/2]} a_{2j-1}(p_{n+1}^{w_{2j}}(x)-p_{n}^{w_{2j-1}}(x)x)E_{2j-1,2j} \\
        & \quad + \sum_{j=1}^{[(N-1)/2]}a_{2j}(p_{n+1}^{w_{2j}}(x)-p_{n}^{w_{2j+1}}(x)x)E_{2j+1,2j}.
    \end{split}
\end{equation*}
 Since $p_{n}^{w_{j}}$ is a monic polynomial, we have that $AP_{n+1}(x) - P_{n}(x)Ax$ is of degree less than or equal to $n$. We write $AP_{n+1}(x) - P_{n}(x)Ax = \sum_{j=0}^{n}C_{j}x^{j}$, with $C_{j} \in \mathcal{M}$. Then, the leading coefficient of $Q_{n}(x)$ is given by $$K_{n} = I + \|P_{n}\|^{2}A^{\ast}\|P_{n-1}\|^{-2}A + C_{n}.$$

We need to show that $K_{n}$ is nonsingular for all $n\geq 0$. First, note that if a matrix $B$ satisfies that $B^{2} =0$, then the matrix $I + B$ is nonsingular because $(I+B)(I-B) = I$. Then, $K_{n}$ is nonsingular if and only if $(I+\|P_{n}\|^{2}A^{\ast})K_{n}(I-C_{n})(I-\|P_{n-1}\|^{-2}A) = I + \|P_{n}\|^{2}A^{\ast} - \|P_{n-1}\|^{-2}A$ is nonsingular.

We have that 
$$ I + \|P_{n}\|^{2}A^{\ast} - \|P_{n-1}\|^{-2}A = \begin{psmallmatrix} 1 & -a_{1}\|p_{n-1}^{w_{1}}\|^{-2} & 0 & 0 & 0 & 0 & \cdots \\
a_{1} \| p_{n}^{w_{2}}\|^{2} & 1 & a_{2}\|p_{n}^{w_{2}}\|^{2} & 0 & 0 & 0 & \cdots \\
0 & -a_{2}\|p_{n-1}^{w_{3}}\|^{-2} & 1 & -a_{3} \| p_{n-1}^{w_{3}}\|^{-2} & 0 & 0 & \cdots \\
0 & 0 & a_{3}\|p_{n}^{w_{4}}\|^{2} & 1 & a_{4} \|p_{n}^{w_{4}}\|^{2} & 0 & \cdots \\
0 & 0 & 0 & -a_{4} \| p_{n-1}^{w_{5}}\|^{-2} & 1 & -a_{5}\|p_{n-1}^{w_{5}}\|^{-2} & \cdots \\
\vdots & \vdots & \vdots & \vdots & \ddots & \ddots & \ddots
\end{psmallmatrix}.$$
By Theorem \ref{det tec} in the Appendix, it follows that 
$$\det( I + \|P_{n}\|^{2}A^{\ast} - \|P_{n-1}\|^{-2}A ) = 1+\sum_{k=1}^{[(N+1)/2]} \sum_{1\leq i_{1}<i_{2}<\ldots<i_{k}\leq N-1}\rho_{i_{1}}\cdots \rho_{i_{k}},$$
where $\rho_{1} = a_{1}^{2}\frac{\|p_{n}^{w_{2}}\|^{2}}{\|p_{n-1}^{w_{1}}\|^{2}}, \rho_{2} = a_{2}^{2} \frac{\|p_{n}^{w_{2}}\|^{2}}{\|p_{n-1}^{w_{3}}\|^{2}}, \rho_{3} = a_{3}^{2} \frac{\|p_{n}^{w_{4}}\|^{2}}{\|p_{n-1}^{w_{3}}\|^{2}}, \ldots, \rho_{i} = a_{i}^{2} \frac{\|p_{n}^{w_{2[(i+1)/2]}}\|^{2}}{\|p_{n-1}^{w_{2[i/2]+1}}\|^{2}}$. Since $\rho_{i} \geq 0$ for all $1 \leq i \leq N-1$, it follows that $\det( I + \|P_{n}\|^{2}A^{\ast} - \|P_{n-1}\|^{-2}A ) \not= 0$. Thus, $K_{n}$ is nonsingular for all $n\geq 0$.

Finally, we need to show the orthogonality. By \eqref{DM} and \eqref{MM}, it follows that 
\begin{equation} \label{QnT}
    Q_{n}(x)(I+Ax) = P_{n}(x) +AP_{n+1}(x) - \|P_{n}\|^{2}A^{\ast}\|P_{n-1}\|^{-2}P_{n-1}(x).
\end{equation}

Let $n\not= m$, we have that 
\begin{equation*}
        \begin{split}
            \langle Q_{n},Q_{m} \rangle_{W} & = \langle Q_{n}T, Q_{m}T \rangle_{\tilde{W}}  \\
            & = \langle P_{n}, P_{n}\rangle_{\tilde{W}}+ A\langle P_{n+1}, P_{m} \rangle_{\tilde{W}}-\langle P_{n}, P_{m-1}\rangle_{\tilde{W}} \|P_{m-1}\|^{-2}A\|P_{m}\|^{2} \\ 
            & \quad + \langle P_{n},P_{m+1} \rangle_{\tilde{W}} A^{\ast}-\|P_{n}\|^{2}A^{\ast}\|P_{n-1}\|^{-2}\langle P_{n-1}, P_{m}\rangle_{\tilde{W}} + A\langle P_{n+1}, P_{m+1} \rangle_{\tilde{W}} A^{\ast} \\
            & \quad  + \|P_{n}\|^{2}A^{\ast}\|P_{n-1}\|^{-2}\langle P_{n-1}, P_{m-1}\rangle_{\tilde{W}} \|P_{m-1}\|^{-2}A\|P_{m}\|^{2}.
        \end{split}
    \end{equation*}

 From here, by the orthogonality of $P_{n}$ with respect to $\langle \cdot \, , \cdot \rangle_{\tilde{W}}$, we obtain that
    \begin{equation} \label{inner}
        \begin{split}
            \langle Q_{n},Q_{m}\rangle_{W} & = A\langle P_{n+1},P_{m} \rangle_{\tilde{W}} - \langle P_{n}, P_{m-1} \rangle_{\tilde{W}}\|P_{m-1}\|^{-2}A\|P_{m}\|^{2} \\
            & \quad + \langle P_{n},P_{m+1} \rangle_{\tilde{W}} A^{\ast} - \|P_{n}\|^{2}A^{\ast}\|P_{n-1}\|^{-2}\langle P_{n-1},P_{m} \rangle_{\tilde{W}}.
        \end{split}
    \end{equation}
To conclude the proof, we need to consider three cases: $n = m+1$, $n = m-1$, and $n \neq m+1$ and $n \neq m-1$. In all three cases, it can be straightforwardly concluded that \eqref{inner} is equal to 0.

\end{proof}

\begin{prop}
    The squared norm of $Q_{n}$ is given by 
    $$\|Q_{n}\|^{2} = \|P_{n}\|^{2} + A\|P_{n+1}\|^{2}A^{\ast} + \|P_{n}\|^{2}A^{\ast}\|P_{n-1}\|^{-2}A\|P_{n}\|^{2}.$$
\end{prop}
\begin{proof}
    We have that 
    $$\|Q_{n}\|^{2} = \langle Q_{n}(I+Ax),Q_{n}(I+Ax)\rangle_{\tilde{W}}.$$
    The result follows directly by computing the inner product and using the expression $Q_{n}(x)(I+Ax) = P_{n}(x) +AP_{n+1}(x) - \|P_{n}\|^{2}A^{\ast}\|P_{n-1}\|^{-2}P_{n-1}(x).$
\end{proof}

Explicitly, the expression of the weight matrix and the sequence of orthogonal polynomials for $N = 2$, and $N = 3$, are given by 

\ 

\textbf{$2\times 2$ cases:}
\begin{equation} \label{W 2}
W(x) = \begin{pmatrix} w_{1}(x)+a^{2}x^{2}w_{2}(x) & axw_{2}(x) \\ 
axw_{2}(x) & w_{2}(x) \end{pmatrix},
\end{equation}
$$Q_{n}(x) = \begin{pmatrix} p^{w_{1}}_{n}(x) & a (p^{w_{2}}_{n+1}(x) - p^{w_{1}}_{n}(x)x) \\ -a \frac{\|p^{w_{2}}_{n}\|^{2}}{\|p^{w_1}_{n-1}\|^{2}}p^{w_{1}}_{n-1}(x) & a^{2}\frac{\|p^{w_{2}}_{n}\|^{2}}{\|p^{w_{1}}_{n-1}\|^{2}}p^{w_{1}}_{n-1}(x)x+p_{n}^{w_{2}}(x)\end{pmatrix}.$$

\

\textbf{$3\times 3$ cases:}
\begin{equation} \label{W 3}
W(x) = \begin{pmatrix}a_{1}^{2}x^{2}w_{2}(x) + w_{1}(x) & a_{1}xw_{2}(x) & a_{1}a_{2}x^{2}w_{2}(x) \\ a_{1}xw_{2}(x) & w_{2}(x) & a_{2}xw_{2}(x) \\ a_{1}a_{2}x^{2}w_{2}(x) & a_{2}xw_{2}(x) & a_{2}^{2}x^{2}w_{2}(x)+w_{3}(x)\end{pmatrix},
\end{equation}
$$Q_{n}(x) = \begin{psmallmatrix}p_{n}^{w_{1}}(x) & a_{1}(p_{n+1}^{w_{2}}(x)-p_{n}^{w_{1}}(x)x) & 0 \\ -a_{1} \frac{\|p_{n}^{w_{2}}\|^{2}}{\|p_{n-1}^{w_{1}}\|^{2}}p^{w_{1}}_{n-1}(x) & a_{1}^{2} \frac{\|p_{n}^{w_{2}}\|^{2}}{\|p_{n-1}^{w_{1}}\|^{2}}p^{w_{1}}_{n-1}(x)x + p_{n}^{w_{2}}(x) +a_{2}^{2} \frac{\|p_{n}^{w_{2}}\|^{2}}{\|p_{n-1}^{w_{3}}\|^{2}}p^{w_{3}}_{n-1}(x)x & -a_{2} \frac{\|p_{n}^{w_{2}}\|^{2}}{\|p_{n-1}^{w_{3}}\|^{2}}p^{w_{3}}_{n-1}(x) \\ 0 & a_{2}(p^{w_{2}}_{n+1}(x)-p_{n}^{w_{3}}(x)x) & p_{n}^{w_{3}}(x)   \end{psmallmatrix}.$$

\section{Bispectrality and Darboux transformation}\label{S4}
This section focuses on the bispectrality of the matrix orthogonal polynomials. We state a sufficient condition on the scalar polynomials to ensure that the sequence of matrix orthogonal polynomials for a weight matrix \( W \), as defined in \eqref{B}, is an eigenfunction of a differential operator. We develop examples constructed from classical scalar weights and obtain new families of bispectral orthogonal polynomials. Additionally, as regards the classification of the Matrix Bochner Problem in \cite{CY18}, we analyze whether the constructed weights arise as Darboux transformations of diagonal scalar weights.

\begin{thm}\label{bisp}
Let $p^{w_i}_n(x)$ be the sequence of monic orthogonal polynomials associated with the scalar weight \( w_i \), for \( i = 1, \dots, N \). Consider the nilpotent matrix \( A \) as defined in \eqref{A} and the polynomial matrix \( T(x) = e^{Ax} \). Assume that the sequence \( p^{w_i}_n(x) \) is an eigenfunction of a differential operator \( \delta_i \), satisfying  
\[
p^{w_i}_n(x) \cdot \delta_i = \Lambda_n(\delta_i) p^{w_i}_n(x),
\]  
with the condition 

\begin{equation} \label{condition}
    \begin{split}
        \Lambda_{n}(\delta_{2i-1}) & = \Lambda_{n+1}(\delta_{2i}), \text{ for all } 1 \leq i \leq [N/2],\\
        \Lambda_{n+1}(\delta_{2i}) & = \Lambda_{n}(\delta_{2i+1}), \text{ for all } 1 \leq i \leq [(N-1)/2].
    \end{split}
\end{equation}
Then, the sequence of matrix-valued orthogonal polynomials \( Q_n(x) \) constructed in Theorem 3.1 for the weight matrix $W(x) = T(x) \operatorname{diag}(w_1, \dots, w_N) T(x)^*$ 
is an eigenfunction of the differential operator $
D = T(x) \operatorname{diag}(\delta_1, \dots, \delta_N) T(x)^{-1}$,
$$Q_{n}(x) \cdot D = \operatorname{diag}(\Lambda_{n}(\delta_{1}),\ldots,\Lambda_{n}(\delta_{N}))Q_{n}(x).$$
\end{thm}

\begin{proof}
    Let $P_{n}(x) = \operatorname{diag}(p_{n}^{w_{1}}(x), \dots, p_{n}^{w_{N}}(x))$,
    and let \( \tilde{D} = \operatorname{diag}(\delta_{1}, \dots, \delta_{N}) \). By hypothesis, the sequence $P_{n}(x)$ satisfy the equation
    \[ P_{n}(x) \cdot \tilde{D} = \Lambda_{n}(\tilde{D}) P_{n}(x), \]
    where \( \Lambda_{n}(\tilde{D}) = \operatorname{diag}(\Lambda_{n}(\delta_{1}), \dots, \Lambda_{n}(\delta_{N})) \). 
    
    Now, we make the following observation. For a nilpotent matrix
    \[ M = \sum_{j=1}^{[ N/2 ]} m_{2j-1}E_{2j-1,2j} + \sum_{j=1}^{[ (N-1)/2 ]} m_{2j}E_{2j+1,2j}, \]
    where \( m_j \neq 0 \) for all \( j \), and for the diagonal matrices \( B = \operatorname{diag}(b_{1}, \dots, b_{N}) \) and \( C = \operatorname{diag}(c_{1}, \dots, c_{N}) \), it follows that \( BM = MC \) if and only if
\begin{equation*} 
        b_{2j-1} = c_{2j}, \text{ for all } 1 \leq j \leq [N/2], \text{ and } c_{2j}  = b_{2j+1}, \text{ for all } 1 \leq j \leq [(N-1)/2].
\end{equation*} 
    Using the identity \eqref{QnT}, we compute
     \begin{equation}\label{aaa}
         \begin{split}
             Q_{n}(x) \cdot T(x)\tilde{D} & = (P_{n}(x) + AP_{n+1}(x) - \|P_{n}\|^{2}A^{\ast}\|P_{n-1}\|^{-2}P_{n-1}(x)) \cdot \tilde{D} \\
             & =\Lambda_{n}(\tilde{D})P_{n}(x) + A\Lambda_{n+1}(\tilde{D})P_{n+1}(x) - \|P_{n}\|^{2}A^{\ast}\|P_{n-1}\|^{-2}\Lambda_{n-1}(\tilde{D})P_{n-1}(x).
         \end{split}
     \end{equation}
    
    Since the hypothesis states that the eigenvalues satisfy \eqref{condition}, it follows from our previous observation that
    \[ A \Lambda_{n+1}(\tilde{D}) = \Lambda_{n}(\tilde{D}) A, \quad \text{and} \quad \|P_{n}\|^{2} A^{\ast} \|P_{n-1}\|^{-2} \Lambda_{n-1}(\tilde{D}) = \Lambda_{n}(\tilde{D}) \|P_{n}\|^{2} A^{\ast} \|P_{n-1}\|^{-2}. \]
    Substituting these into \eqref{aaa}, we obtain
    \[ Q_{n}(x) \cdot T(x) \tilde{D} = \Lambda_{n}(\tilde{D}) Q_{n}(x) T(x), \]
    which completes the proof.
\end{proof}

In the following subsections, we use the Theorem \ref{bisp} to construct bispectral matrix orthogonal polynomials with the classical scalar polynomials of Laguerre, Hermite, and Jacobi.

\subsection{Laguerre-type}
Let $\alpha_{1},\ldots,\alpha_{N} > -1$. We consider the scalar Laguerre weights $w_{i}(x) = e^{-x}x^{\alpha_{i}}$, and we construct the weight matrix $W$ as in \eqref{B},
$$W(x) = T(x)\operatorname{diag}(e^{-x}x^{\alpha_{1}},\ldots, e^{-x}x^{\alpha_{N}})T(x)^{\ast}.$$
Let $\delta_{\alpha} = \partial^{2}x + \partial(\alpha+1-x)$ be the second-order differential operator of Laguerre. The eigenvalue of $\delta_{\alpha}$ is $\Lambda_{n}(\delta_{\alpha}) = -n$. By considering the operators \( \delta_{\alpha_{2j-1}} \) and \( \delta_{\alpha_{2j}} + 1 \), the eigenvalue condition \eqref{condition} is satisfied. Thus, applying Theorem \ref{bisp}, it follows that \( Q_n(x) \) is an eigenfunction of the second-order differential operator $
D = T(x) \tilde{D} T(x)^{-1}$, where $\tilde{D} = \operatorname{diag}(\delta_{\alpha_{1}},\ldots,\delta_{\alpha_{N}}) + \sum_{j=0}^{[N/2]}E_{2j,2j}$. This proves that \( Q_n(x) \) is bispectral.
The explicit expression of $D$ is given by
$$D = \partial ^2 x I+ \partial \Big( B_L-xI+2xA+x[A,B_L]\Big) + AB_L+ K,$$
where
$$B_L=\textstyle \sum_{j=1}^{N} (\alpha_j+1) E_{j,j},  \, \text{ and } \, K=\sum_{j=1}^{[N/2]} E_{2j,2j}.$$

Explicitly, for $N = 2$, we take the scalar Laguerre weights $w_{\alpha}(x) = e^{-x}x^{\alpha}$, and $w_{\beta}(x) = e^{-x}x^{\beta}$, for $\alpha,\beta > -1$. We have the weight matrix $W$ given by 
$$W(x) = e^{-x}\begin{pmatrix}x^{\alpha}+a^{2}x^{\beta+2} & ax^{\beta+1} \\ ax^{\beta+1} & x^{\beta}\end{pmatrix}.$$
The sequence of orthogonal polynomials $Q_{n}$ for $W$ is given by
$$Q_{n}(x) = \begin{pmatrix} \ell_{n}^{(\alpha)}(x) & a (\ell_{n+1}^{(\beta)}(x) - \ell_{n}^{(\alpha)}(x)x) \\ -a n \frac{\Gamma(n+\beta+1)}{\Gamma(n+\alpha)}\ell_{n-1}^{(\alpha)}(x) & a^{2} n \frac{\Gamma(n+\beta+1)}{\Gamma(n+\alpha)}\ell_{n-1}^{(\alpha)}(x)x + \ell^{(\beta)}_{n}(x)\end{pmatrix}$$
where $\ell_{n}^{(\alpha)}(x)$ and $\ell_{n}^{(\beta)}(x)$ are the sequence of monic orthogonal polynomials for $w_{1}(x)$ and $w_{2}(x)$, respectively. We have that $Q_{n}(x)$ is eigenfunction of the second-order differential operator
$$D = \partial^{2} xI + \partial \begin{pmatrix}\alpha + 1 -x&& ax(2+\beta-\alpha) \\ 0 && \beta + 1 -x  \end{pmatrix} + \begin{pmatrix} 0 && a(\beta+1) \\ 0 && 1 \end{pmatrix},$$
with
$$Q_{n}(x) \cdot D = \begin{pmatrix} -n & 0 \\ 0 & -n+1\end{pmatrix}Q_{n}(x).$$
The sequence $Q_{n}$ satisfies the three-term recurrence relation 
$Q_{n}(x)x = A_{n}Q_{n+1}(x) + B_{n}Q_{n}(x) + C_{n}Q_{n-1}(x),$ where
\begin{equation*}
    \begin{split}
        A_{n} & = \begin{psmallmatrix} 1 & \frac{a(n+\alpha)(\beta-\alpha+2)}{g_{n}(n+\beta+1)a^{2}(n+1)+\alpha+n} \\ 0 &  \frac{a^{2}g_{n}(n+\alpha)n+\alpha+n}{g_{n}(n+\beta+1)a^{2}(n+1)+\alpha+n}\end{psmallmatrix}, \quad C_{n} = \begin{psmallmatrix} \frac{n(g_{n}a^{2}(n+1)(n+\beta+1)+\alpha+n)}{g_{n}a^{2}n+1} & 0 \\ \frac{g_{n}an(\beta-\alpha+2)}{g_{n}a^{2}n+1} & n(n+\beta)\end{psmallmatrix}, \\
        B_{n} & = \begin{psmallmatrix} \frac{g_{n}a^{2}(n+1)(2n+\beta+3)(n+\beta+1)+(2n+\alpha+1)(n+\alpha)}{g_{n}a^{2}(n+1)(n+\beta+1)+\alpha+n} & \frac{a(1+\beta+n(\beta-\alpha+2))}{g_{n}a^{2}n+1} \\ \frac{g_{n}a(1+\beta+n(\beta-\alpha+2))}{g_{n}a^{2}(n+1)(n+\beta+1)+n+\alpha} & \frac{g_{n}a^{2}n(2n+\alpha-1)+\beta+2n+1}{g_{n}a^{2}n+1}\end{psmallmatrix},
    \end{split}
\end{equation*}
where $g_{n} = \frac{\Gamma(n+\beta+1)}{\Gamma(n+\alpha)}$.

\

As regards the Darboux transformation, it was shown in \cite{BP24-1} that if $\alpha_{i} - \alpha_{j} \notin \mathbb{Z}$ for all $i\not= j$, then the weight $W(x) = T(x)\operatorname{diag}(e^{-x}x^{\alpha_{1}},\ldots, e^{-x}x^{\alpha_{N}})T(x)^{\ast}$ is not a Darboux transformation of any diagonal of scalar weights.

On the other hand, at the other extreme, if $\alpha_{i} - \alpha_{j} \in \mathbb{Z}$ for all $i,j$, then the weight $W$ is a Darboux transformation of a direct sum of scalar Laguerre weights, as we will demonstrate in Theorem \ref{darb lag a}. We first establish the following proposition, which will be essential for establishing the result.

\begin{prop}\label{Lag identities}
    Let $r_{1}$, $r_{2}$ be scalar polynomials. Let $k \in \mathbb{Z}$, $m \in \mathbb{Z}$, $\alpha > -1$. Let $\ell_{n}^{(\alpha)}(x)$ be the sequence of monic orthogonal polynomials for the Laguerre weight $w_{\alpha}(x) = e^{-x}x^{\alpha}$. Then, there exists a polynomial $q$ and a differential operator $\tau$ such that 
    $$\ell_{n}^{(\alpha)} \cdot \tau = q(n) \frac{r_{1}(n)}{r_{2}(n)}\ell_{n+m}^{(\alpha+k)}(x).$$
\end{prop}
\begin{proof}
We have that the monic orthogonal polynomials of Laguerre satisfy the following identities
    \begin{equation*}
        \begin{split}
           \ell_{n}^{(\alpha)}(x)\cdot (\partial - 1) & = \ell_{n}^{(\alpha+1)}(x), \quad \ell_{n}^{(\alpha)}(x)\cdot (\partial x + \alpha) = (n+\alpha)\ell_{n}^{(\alpha-1)}(x), \\ \ell_{n}^{(\alpha)}(x) \cdot (\partial-1)(x-\partial x + \alpha + 1)  & = \ell_{n+1}^{(\alpha)}(x), \quad  \ell_{n}^{(\alpha)}(x) \cdot \partial = n\ell_{n-1}^{(\alpha+1)}(x), \quad
            \ell_{n}^{(\alpha)}(x)\cdot \delta_{\alpha}  = -n \ell_{n}^{(\alpha)}(x),
        \end{split}
    \end{equation*}
    where $\delta_{\alpha} = \partial^{2}x + \partial (\alpha+1-x)$ is the second-order differential operator of Laguerre.
    By composing the differential operators above, we can construct a differential operator $\tilde{\tau}$ such that 
    $$\ell_{n}^{(\alpha)}(x) \cdot \tilde{\tau} = \tilde{q}(n)\ell_{n+m}^{(\alpha+k)}(x),$$
    for some polynomial $\tilde{q}$. By taking $q(x) = r_{2}(x)\tilde{q}(x)$ and $\tau = r_{1}(-\delta_{\alpha})\tilde{\tau}$ the statement holds.
\end{proof}

\begin{thm}\label{darb lag a}
    If $\alpha_{i}-\alpha_{j} \in \mathbb{Z}$ for all $i,j$, then $W(x) = T(x)\operatorname{diag}(e^{-x}x^{\alpha_{1}},\ldots, e^{-x}x^{\alpha_{N}})T(x)^{\ast}$ is a Darboux transformation of the diagonal of scalar weights $\tilde{W}(x) = \operatorname{diag}(w_{\alpha_{1}(x)},\ldots,w_{\alpha_{N}}(x))$.
\end{thm}
\begin{proof}
    
We have the sequence of orthogonal polynomials $Q_{n}$ for $W$ given by Theorem \ref{main}. From \eqref{QnT} we have that 
$$Q_{n}(x)T(x) = \begin{psmallmatrix} \ell_{n}^{(\alpha_{1})}(x) & a_{1} \ell_{n+1}^{(\alpha_{2})}(x) & 0 & 0 & \ldots \\ 
-a_{1} \frac{\|\ell_{n}^{(\alpha_{2})}\|^{2}}{\|\ell_{n-1}^{(\alpha_{1})}\|^{2}} \ell_{n-1}^{(\alpha_{1})}(x) & \ell_{n}^{(\alpha_{2})}(x) & -a_{2} \frac{\|\ell_{n}^{(\alpha_{2})}\|^{2}}{\|\ell_{n-1}^{(\alpha_{3})}\|^{2}} \ell_{n-1}^{(\alpha_{3})}(x) & 0 & \ldots \\ 0 & a_{2} \ell_{n+1}^{(\alpha_{2})}(x) & \ell_{n}^{(\alpha_{3})}(x) & a_{3}\ell_{n+1}^{(\alpha_{4})}(x) & \ldots \\ 0 & 0 & -a_{3} \frac{\|\ell_{n}^{(\alpha_{4})}\|^{2}}{\|\ell_{n-1}^{(\alpha_{3})}\|^{2}}\ell_{n-1}^{(\alpha_{3})}(x) & \ell_{n}^{(\alpha_{4})}(x) & \ldots \\
\vdots & \vdots & \vdots & \vdots & \ddots\end{psmallmatrix} = \begin{pmatrix}\vec{v}_{1} \\ \vec{v}_{2} \\ \vec{v}_{3} \\ \vdots \\ \vec{v}_{N} \end{pmatrix}.$$

The rows of $Q_{n}(x)T(x)$ satisfy that 
\begin{equation*}
    \begin{split}
        \vec{v}_{1} & = \begin{psmallmatrix} \ell_{n}^{(\alpha_{1})}(x) && a_{1} \ell_{n+1}^{(\alpha_{2})}(x) && 0 && \ldots && 0 \end{psmallmatrix} \\
        \vec{v}_{2i} & = \begin{psmallmatrix} 0 && \ldots && 0 &&  -a_{2i-1}\frac{\| \ell_{n}^{(\alpha_{2i})}\|^{2}}{\| \ell_{n-1}^{(\alpha_{2i-1})} \|^{2}}
     \ell_{n-1}^{(\alpha_{2i-1})}(x) && \ell_{n}^{(\alpha_{2i})}(x) &&  -a_{2i}\frac{\| \ell_{n}^{(\alpha_{2i})}\|^{2}}{\| \ell_{n-1}^{(\alpha_{2i+1})} \|^{2}}
     \ell_{n-1}^{(\alpha_{2i-1})}(x) && 0 && \ldots && 0 \end{psmallmatrix}, \\ 
     \vec{v}_{2i+1} & = \begin{psmallmatrix} 0 && \ldots && 0 && a_{2i} \ell_{n+1}^{(\alpha_{2i})}(x) && \ell_{n}^{(\alpha_{2i+1})}(x) && a_{2i+1}\ell_{n+1}^{(\alpha_{2i+2})}(x) && 0 && \ldots && 0 \end{psmallmatrix}, \\
     \vec{v}_{N} & = \left\{ \begin{array}{lcc} \begin{psmallmatrix} 0 && \ldots && 0 && -a_{N-1}\frac{\| \ell_{n}^{(\alpha_{N})}\|^{2}}{\|\ell_{n-1}^{(\alpha_{N-1})}\|^{2}}\ell_{n-1}^{(\alpha_{N-1})}(x) && \ell_{n}^{(\alpha_{N})}(x)\end{psmallmatrix}, & \text{if } N \text{ is an even number,}   \\ \\ 
     \begin{psmallmatrix} 0 && \ldots && 0 && a_{N-1}\ell_{n+1}^{(\alpha_{N-1})}(x) && \ell_{n}^{(\alpha_{N})}(x) \end{psmallmatrix}, & \text{if } N \text{ is an odd number.}\end{array} \right.
    \end{split}
\end{equation*}
Since $\alpha_{i} - \alpha_{j} \in \mathbb{Z}$, it follows that $\frac{\|\ell_{n}^{(\alpha_{i})}\|^{2}}{\|\ell_{n-1}^{(\alpha_{j})}\|^{2}}$ is a rational function for all $i,j$. Hence, by Proposition \ref{Lag identities}, we have that there exist polynomials $q_{i}$ and operators $\tau_{i,j}$, $i = 1,\ldots,N$, $j = 1, \ldots,3$ such that 
\begin{equation*}
    \ell_{n}^{(\alpha_{i})}(x) \cdot \begin{pmatrix}0 & \ldots & 0 & \tau_{i,1} & \tau_{i,2} & \tau_{i,3} & 0 & \ldots & 0 \end{pmatrix} = q_{i}(n)\vec{v}_{i}.
\end{equation*}

By arranging these differential operators as rows, we construct a differential operator $\tilde{\mathcal{D}}_{1}$ that satisfies 
$$\operatorname{diag}(\ell_{n}^{(\alpha_{1})}(x), \ldots, \ell_{n}^{(\alpha_{N})}(x)) \cdot \tilde{\mathcal{D}}_{1}T(x)^{-1} = \operatorname{diag}(q_{1}(n),\ldots,q_{N}(n)) Q_{n}(x).$$
Thus, by Proposition \ref{darb eq}, the statement holds.
\end{proof}

\

To illustrate the above situation, we give an explicit example for $N = 5$. We consider the direct sum of classical scalar Laguerre weights
$$\tilde{W}(x) = w_{\alpha}(x) \oplus w_{\alpha}(x) \oplus w_{\alpha+1}(x) \oplus w_{\alpha+1}(x) \oplus w_{\alpha+2}(x).$$
We have the nilpotent matrix
$A = \begin{pmatrix} 0 & a_{1} & 0 & 0 & 0 \\ 0 & 0 & 0 & 0 & 0 \\
0 & a_{2} & 0 & a_{3} & 0 \\
0 & 0 & 0 & 0 & 0 \\
0 & 0 & 0 & a_{4} & 0
\end{pmatrix}$, $a_{1},a_{2},a_{3},a_{4} \in \mathbb{R} - \{ 0 \}$. The weight matrix $W(x) = T(x)\tilde{W}(x)T(x)^{\ast}$, with $T(x) = e^{Ax}$, is given explicitly by

\begin{equation}\label{ej Lag 5}
    W(x) = e^{-x}x^{\alpha}  \left( \begin {array}{ccccc} {a_{1}}^{2}{x}^{2}+1&a_{1}\,x&{x}^
    {2}a_{1}\,a_{2}&0&0\\ \noalign{\medskip}a_{1}\,x&1&a_{2}\,
    x&0&0\\ \noalign{\medskip}{x}^{2}a_{1}\,a_{2}&a_{2}\,x&{{\it 
    a_{3}}}^{2}{x}^{3}+{a_{2}}^{2}{x}^{2}+x&{x}^{2}a_{3}&{x}^{3}a_{3}
    \,a_{4}\\ \noalign{\medskip}0&0&{x}^{2}a_{3}&x&{x}^{2}a_{4}
    \\ \noalign{\medskip}0&0&{x}^{3}a_{3}\,a_{4}&{x}^{2}a_{4}&{a_{4}}^{2}{x}^{3}+{x}^{2}\end {array} \right).
\end{equation}
The sequence of monic orthogonal polynomials for $\tilde{W}$ is given by 
$$P_{n}
(x) = \operatorname{diag}(\ell_{n}^{(\alpha)}(x),\ell_{n}^{(\alpha)}(x),\ell_{n}^{(\alpha+1)}(x),\ell_{n}^{(\alpha+1)}(x),\ell_{n}^{(\alpha+2)}(x)).$$

Let $Q_{n}(x)$ be the sequence of orthogonal polynomials for $W$ given by Theorem \ref{main}. From \eqref{QnT}, we have that
$$Q_{n}(x)T(x) = \begin{psmallmatrix} \ell_{n}^{(\alpha)}(x) & a_{1} \ell_{n+1}^{(\alpha)}(x) & 0 & 0 & 0 \\
-a_{1}n(n+\alpha)\ell_{n-1}^{(\alpha)}(x) & \ell_{n}^{(\alpha)}(x) & -na_{1}\ell_{n-1}^{(\alpha+1)}(x) & 0 & 0 \\ 
0 & a_{2} \ell_{n+1}^{(\alpha)}(x) & \ell_{n}^{(\alpha+1)}(x) & a_{3} \ell_{n+1}^{(\alpha+1)}(x) & 0  \\
0 & 0 & -a_{3}n(n+\alpha+1)\ell_{n-1}^{(\alpha+1)}(x) & \ell_{n}^{(\alpha +1)}(x) & -na_{4}\ell_{n-1}^{(\alpha+2)}(x) \\ 
0 & 0 & 0 & a_{4}\ell_{n+1}^{(\alpha+1)}
(x) & \ell_{n}^{(\alpha+2)}\end{psmallmatrix}.$$
The differential operator 
$$\tilde{\mathcal{D}_{1}} = \begin{psmallmatrix}1 & a_{1}(\partial^{2}x + \partial (\alpha+1-2x) + (x - \alpha - 1)) & 0 & 0 & 0 \\ -a_{1} (\partial^{2} x + \partial (\alpha + 1)) & 1 & -a_{2}\partial & 0 & 0 \\ 0 & -a_{2}(\partial x -(x - \alpha - 1)) & 1 & a_{3}(\partial^{2}x + \partial (\alpha + 2 - 2x ) + (x - \alpha - 2)) & 0 \\ 0 & 0 & -a_{3}(\partial^{2} x + \partial (\alpha + 2)) & 1 & -a_{4} \partial \\ 0 & 0 & 0 & -a_{4} (\partial x - (x- \alpha - 2)) & 1 \end{psmallmatrix}$$
satisfies that 
$$P_{n}(x) \cdot \tilde{\mathcal{D}_{1}} = Q_{n}(x)T(x).$$
Then, by taking $\mathcal{D}_{1} = \tilde{\mathcal{D}}_{1}T(x)^{-1}$ we have $P_{n}(x) \cdot \mathcal{D}_{1} = Q_{n}(x).$
Thus, by Proposition \ref{darb eq}, we have that $W$ is a Darboux transformation of the direct sum of scalar weights $\tilde{W}$. 

In particular, the explicit expression of the operator $D$ that is factored in Definition \ref{darb def} is
\begin{equation*}
    \begin{split}
        D & = \begin{psmallmatrix}a_{1}^{2}\delta_{\alpha}^{2}-(a_{1}^{2}(\alpha+2))(\delta_{\alpha}-1) & 0 & a_{1}a_{2}(\partial - 1)(\delta_{\alpha+1} - 1) & 0 & 0 \\
        0 & 0 & 0 & 0 & 0 \\ a_{1}a_{2}(\partial x + \alpha + 1)(1-\delta_{\alpha}) & 0 & a_{3}^{2}\delta_{\alpha+1}^{2}-(a_{3}^{2}(\alpha+3)+a_{2}^{2})(\delta_{\alpha+1}-1) & 0 & a_{3}a_{4}(\partial - 1)(\delta_{\alpha+2}-1) \\ 0 & 0 & 0 & 0 & 0 \\ 0 & 0 & a_{3}a_{4}(\partial x + \alpha + 2)(1-\delta_{\alpha+1}) & 0 & a_{5}^{2}\delta_{\alpha+2}^{2}-(a_{5}^{2}(\alpha+4)+a_{4}^{2})(\delta_{\alpha+2}-1)  \end{psmallmatrix} \\
        & \quad + \begin{psmallmatrix} 0 & 0 & 0 & 0 & 0 \\ 0 & a_{2}^{2}\delta_{\alpha}^{2}-(a_{1}^{2}\alpha + a_{2}^{2})\delta_{\alpha}+1 & 0 &  a_{2
}a_{3} (\partial - 1)\delta_{\alpha+1} & 0 \\ 0 & 0 & 0 & 0 & 0 \\ 0 & -a_{2}a_{3}(\partial x + \alpha + 1)\delta_{\alpha} & 0 & a_{4}^{2}\delta_{\alpha+1}^{2}-(a_{3}^{2}(\alpha+1)+a_{4}^{2}) \delta_{\alpha+1} +1 & 0  \\ 0 & 0 & 0 & 0 & 0 \end{psmallmatrix},
    \end{split}
\end{equation*}
where $\delta_{\alpha} = \partial^{2}x + \partial (\alpha+1-x)$ is the second-order differential operator of the Laguerre weight of parameter $\alpha$.
We have that the operator $D$ belongs to $\mathcal{D}(\tilde{W})$ (see Theorem 6.7, \cite{BPdarb}), and is factored as $D = \mathcal{D}_{1}\mathcal{D}_{2}$, with $\mathcal{D}_{2} = T(x)(2I-\tilde{\mathcal{D}}_{1})$.

\subsection{Hermite-type}
Let $c_{1},\ldots, c_{N}$ be real numbers. For $i=1,\ldots,N$, we consider the shifted Hermite weight $w_{i}(x) = e^{-x^{2}+2c_{i}x}$. We construct the weight matrix $W$ as in \eqref{B},
$$W(x) = T(x)\operatorname{diag}(e^{-x^{2}+2c_{1}x},\ldots,e^{-x^{2}+2c_{N}x})T(x)^{\ast}.$$
Let $\delta_{c} = \partial^{2} + \partial (-2(x-c))$ be the second-order differential operator for the Hermite weight $w(x) = e^{-x^{2}+2cx}$. The eigenvalue of $\delta_{c}$ is $\Lambda_{n}(\delta_{c}) = -2n$. It follows that $\Lambda_{n}(\delta_{c_{2i-1}}) = \Lambda_{n+1}(\delta_{c_{2i}}+2)$. Therefore, by Theorem \ref{bisp}, we have that the sequence of orthogonal polynomials $Q_{n}(x)$ for $W$ given by Theorem \ref{main} is an eigenfunction of the second-order differential operator $D = T(x)\tilde{D}T(x)^{-1}$, with $\tilde{D} = \operatorname{diag}(\delta_{c_{1}},\ldots, \delta_{c_{N}}) + \sum_{j=1}^{[N/2]}-2E_{2j,2j}.$ Thus, $Q_{n}(x)$ is a bispectral function. 

The differential operator $D$ is given by
$$D = \partial ^2 I+ \partial \Big( 2A+B_H-x+x[A,B_H]\Big) + AB_H+ 2K,$$
where

$$B_H= \textstyle \sum_{j=1}^N 2c_j E_{j,j}, \, \text{ and } \, K=\sum_{j=1}^{[N/2]} E_{2j,2j}.$$

Explicitly, for $N = 2$, we have, $w_{1}(x) = e^{-x^{2}+2bx}$ and $w_{2}(x) = e^{-x^{2}+2cx}$, with $b,c \in \mathbb{R}$. The weight matrix is given by
$$W(x) = e^{-x^{2}}\begin{pmatrix} e^{2bx}+a^{2}x^{2}e^{2cx} & axe^{2cx} \\ axe^{2cx} & e^{2cx}\end{pmatrix},$$
and from Theorem \ref{main} a sequence of orthogonal polynomials for $W$ is given by
$$Q_{n}(x) = \begin{pmatrix} h_{n}(x-b) & a (h_{n+1}(x-c)-h_{n}(x-b)x) \\ -a n\frac{e^{c^{2}-b^{2}}}{2} h_{n-1}(x-b) & a^{2} n \frac{e^{b^{2}-c^{2}}}{2} h_{n-1}(x-b)x + h_{n}(x-c)\end{pmatrix},$$
where $h_{n}(x)$ is the $n$-th monic orthogonal polynomials of the Hermite weight $e^{-x^{2}}$. 
We have that the sequence $Q_{n}$ is an eigenfunction of the second-order differential operator
$$D = \partial^{2} I + \partial \begin{pmatrix} -2x+2b & 2ax(c-b)+2a \\ 0 & -2x +2c\end{pmatrix} + \begin{pmatrix} -2 & 2ac \\ 0 & 0\end{pmatrix},$$
with
$$Q_{n}(x) \cdot D = \begin{pmatrix} -2n-2 & 0 \\ 0 & -2n\end{pmatrix}Q_{n}(x).$$
And, we have that the sequence $Q_{n}$ satisfies the three-term recurrence relation

\begin{equation*}
    \begin{split}
        Q_{n}(x)x & = \begin{psmallmatrix} 1 & - \frac{2ab}{e^{c^{2}-b^{2}}a^{2}(n+1)+2} \\ 0 & \frac{e^{c^{2}-b^{2}}a^{2}n+2}{e^{c^{2}-b^{2}}a^{2}(n+1)+2}\end{psmallmatrix}Q_{n+1}(x) + \begin{psmallmatrix} \frac{2b}{e^{c^{2}-b^{2}}a^{2}(n+1)+2} & \frac{a}{e^{c^{2}-n^{2}}a^{2}n+2} \\ \frac{e^{c^{2}-b^{2}}a}{e^{c^{2}-b^{2}}a^{2}(n+1)+2} & \frac{e^{c^{2}-b^{2}}a^{2}nb+2c}{e^{c^{2}-b^{2}}a^{2}n+2}\end{psmallmatrix}Q_{n}(x) \\
        & \quad + \begin{psmallmatrix} \frac{n(e^{c^{2}-b^{2}}a^{2}(n+1)+2)}{e^{c^{2}-b^{2}}a^{2}n+2} & 0 \\ \frac{ae^{c^{2}-b^{2}}n(c-b)}{e^{c^{2}-b^{2}}a^{2}n+2} & \frac{n}{2}\end{psmallmatrix}Q_{n-1}(x).
    \end{split}
\end{equation*}

An important result proved in \cite{BP24-1}, is that if $c_{i}\not = c_{j}$ for all $i \not = j$, then the weight matrix $W(x) = T(x) \operatorname{diag}(e^{-x^{2}+2c_{1}x},\ldots,e^{-x^{2}+2c_{N}x})T(x)^{\ast}$ is not a Darboux transformation of any direct sum of scalar polynomials.

On the other hand,  if $w_{1}(x) = \cdots = w_{N}(x) = e^{-x^{2}}$,
by Theorem \ref{main}, a sequence of orthogonal polynomials for the $N\times N$ Hermite-type weight $W(x) = T(x)e^{-x^{2}}T(x)^{\ast}$ is given by
$$Q_{n}(x) = h_{n}(x)I+Ah_{n+1}(x) - \frac{n}{2}A^{\ast}h_{n-1}(x) - h_{n}(x)Ax + \frac{n}{2}h_{n-1}(x)A^{\ast}Ax,$$
where $h_{n}(x)$ is the sequence of the monic orthogonal polynomials for the scalar Hermite weight.
We have that $W$ is a Darboux transformation of the direct sum of scalar Hermite weights $e^{-x^{2}}I$. 
In fact, the differential operator 
$$D = (\partial^{2} + \partial (-2x))\frac{-(AA^{\ast} + A^{\ast}A)}{4} + \frac{AA^{\ast}}{2} + I \in \mathcal{D}(e^{-x^{2}}I)$$
can be factored as $D = \mathcal{D}_{1}\mathcal{D}_{2}$, with
\begin{equation*}
    \mathcal{D}_{1} = \partial \left ( \frac{A^{\ast}A}{2}x - \frac{A + A^{\ast}}{2} \right ) + I \text{ and } \mathcal{D}_{2} = \partial \left ( \frac{AA^{\ast}}{2}x + \frac{A + A^{\ast}}{2} \right ) + \frac{AA^{\ast}}{2} + I.
\end{equation*}
Furthermore, $Q_{n}(x) = h_{n}(x) \cdot \mathcal{D}_{1}$. 
As a consequence, the differential operator 
$$\mathcal{D} = \mathcal{D}_{2}\mathcal{D}_{1} = \partial^{2} \left ( - \frac{(A^{\ast}A+AA^{\ast})}{4} \right ) + \partial \left ( \frac{AA^{\ast} + A^{\ast}A}{2}x - \frac{AA^{\ast}A}{2} \right ) + \frac{AA^{\ast}}{2} + I$$
belongs to $\mathcal{D}(W)$.

\subsection{Jacobi-type}
For $i = 1, \ldots, N$, let $\alpha_{i}, \, \beta_{i} > -1$. We take the scalar Jacobi weights $w_{\alpha_{i},\beta_{i}}(x) = (1-x)^{\alpha_{i}}(1+x)^{\beta_{i}}$. 
We have the weight matrix 
$$W(x) = T(x) \operatorname{diag}(w_{\alpha_{1},\beta_{1}}(x),\ldots, w_{\alpha_{N},\beta_{N}}(x))T(x)^{\ast},$$
as defined in \eqref{B}. Let $\delta_{\alpha,\beta} = \partial^{2} (1-x^{2}) + \partial (\beta - \alpha - x(\alpha + \beta +2))$ be the second-order differential operator of Jacobi of parameters $\alpha$ and $\beta$. The eigenvalue of $\delta_{\alpha,\beta}$ is $\Lambda_{n}(\delta_{\alpha,\beta}) = -n(n+\alpha+\beta+1)$. If the parameters $\alpha_{i}, \beta_{i}$ satisfy that $\alpha_j+ \beta_j +1+(-1)^j =\alpha_1+\beta_1$, then the eigenvalues of the operators $\delta_{\alpha_{2i-1},\beta_{2i-1}}$ and $\delta_{\alpha_{2i},\beta_{2i}} + \alpha_{1} + \beta_{1}$ satisfy the condition \eqref{condition}. Then, by Theorem \ref{bisp}, $Q_{n}(x)$ is eigenfunction of the second-order differential operator $D = T(x)\tilde{D}T(x)^{-1}$, with $\tilde{D} = \operatorname{diag}(\delta_{\alpha_{1},\beta_{1}},\ldots,\delta_{\alpha_{N},\beta_{N}}) + \sum_{j=1}^{[N/2]}(\alpha_{1} + \beta_{1})E_{2j,2j}$. The differential operator $D$ is given by
$$D  = \partial ^2 (1-x^2) I+ \partial \Big(2A+B_J + x\big( [A,B_J]+2 K -(\alpha_1+\beta_1+2)I\big)\Big) + AB_J+(\alpha_1+\beta_1)K\,$$
with $B_J=\textstyle \sum_{j=1}^{N} (\beta_{j}-\alpha_j) E_{j,j}, \, \text{ and } \, K=\sum_{j=1}^{[N/2]} E_{2j,2j}.$ 
Thus, with these extra hypotheses on the parameters $\alpha_{i}, \beta_{i}$, we have that $Q_{n}(x)$ is bispectral. 

\ 

Regarding the Darboux transformation, it is proven in \cite{BP24-1} that if $\alpha_{i} - \alpha_{j} \notin \mathbb{Z}$ or $\beta_{i} - \beta_{j} \notin \mathbb{Z}$ for all $i \neq j$, then $W(x) = T(x)\operatorname{diag}(w_{\alpha_{1},\beta_{1}}(x),\ldots, w_{\alpha_{N},\beta_{n}}(x))T(x)^{\ast}$ is not a Darboux transformation of any diagonal scalar weight. On the other hand, we did not find any parameters $\alpha_{i}$ and $\beta_{i}$ such that the weight $W$ is a Darboux transformation of some diagonal of scalar weights.

\

Explicitly, for $N = 2$, we have for the Gegenbauer scalar weights $w_{1}(x) = (1-x^{2})^{r+1}$, $w_{2}(x) = (1-x^{2})^{r}$, $r > -1$, that the weight $W$  is
$$W(x) = (1-x)^{r}\begin{pmatrix}1-x^{2}+ax^{2} & ax \\ ax & 1\end{pmatrix}.$$

From Theorem \ref{main}, a sequence of orthogonal polynomials with respect to $W$ is given by
$$Q_{n}(x) = \begin{pmatrix} p_{n}^{(r+1)}(x) & a(p_{n+1}^{(r)}-p_{n}^{(r+1)}(x)x) \\ 
-a\frac{n}{(n+2r+1)}p_{n-1}^{(r+1)}(x) & a^{2}\frac{n}{(n+2r+1)}p_{n-1}^{(r+1)}(x)x+p_{n}^{(r)} \end{pmatrix},$$
where $p_{n}^{(r+1)}$ and $p_{n}^{(r)}$ are the sequence of monic orthogonal polynomials for $w_{1}$ and $w_{2}$, respectively. 
The sequence $Q_{n}$ is eigenfunction of the second-order differential operator 
$$D= \partial^{2} (1-x^{2})I + \partial  \begin{pmatrix} -x(2r + 4) && 2a \\ 0 && - x(r + 2)  
\end{pmatrix} + \begin{pmatrix} 0 && 0 \\ 0 && 2r+2  \end{pmatrix},$$
we have that 
$$Q_{n}(x) \cdot D = \begin{pmatrix}-n(n+2r+3) & 0 \\0 & -(n-1)(n+2r+2) \end{pmatrix} Q_{n}(x).$$
It also satisfies the three-term recurrence relation 
$$Q_{n}(x)x = A_{n}Q_{n+1}(x)+B_{n}Q_{n}(x)+C_{n}Q_{n-1}(x),$$
where
\begin{equation*}
    \begin{split}
         A_{n} & = \begin{psmallmatrix}
             1 & 0 \\ 0 & \frac{ \left( n
+2\,r+2 \right)  \left( n{a}^{2}+n+2\,r+1 \right) }{ \left( n+2\,r+1
 \right)  \left( n{a}^{2}+{a}^{2}+n+2\,r+2 \right)} \end{psmallmatrix}, \\
 B_{n} & = \begin{psmallmatrix} 0 &  -{\frac { \left( n+2\,r+1 \right) a
 \left( {n}^{2}-2\,{r}^{2}+n-r \right) }{ \left( n+r+1 \right) 
 \left( 2\,n+2\,r+3 \right)  \left( 2\,n+2\,r+1 \right)  \left( n{a}^{
2}+n+2\,r+1 \right) }} \\ {\frac { \left( 2\,r+1
 \right) a}{ \left( n+2\,r+1 \right)  \left( n{a}^{2}+{a}^{2}+n+2\,r+2
 \right) }} & 0 \end{psmallmatrix} \\
 C_{n} & = 
         \begin{psmallmatrix} {\frac {n \left( {a}^{2}{n}^{2}+{a}^{2}nr+n
{a}^{2}+{a}^{2}r+{n}^{2}+3\,nr+2\,{r}^{2}+3\,n+4\,r+2 \right)  \left( 
n+2\,r+1 \right) }{ \left( n+r+1 \right)  \left( 2\,n+2\,r+3 \right) 
 \left( 2\,n+2\,r+1 \right)  \left( n{a}^{2}+n+2\,r+1 \right) }} & 0 \\ 0 & {\frac { \left( n+2\,r \right) n}{ \left( 2\,n+
2\,r+1 \right)  \left( 2\,n+2\,r-1 \right) }} \end{psmallmatrix}.
    \end{split}
\end{equation*}

\subsection{Hermite-Laguerre-type}
 Theorem \ref{main} allows us to construct an innovative sequence of matrix-valued orthogonal polynomials by combining Hermite and Laguerre polynomials. Let $\alpha > -1$, we consider the scalar weights of Hermite $w_{1}(x) = e^{-x^{2}}$, and Laguerre $w_{2}(x) = e^{-x}x^{\alpha}$. We construct the positive semi-definite weight matrix $W(x)$ as in \eqref{B}, 
$$W(x) = T(x) \tilde{W}(x)T(x)^{\ast} =\begin{pmatrix} e^{-x^{2}} + a^{2}x^{2}e^{-x}x^{\alpha}1_{(0,\infty)}(x) & ax e^{-x}x^{\alpha} 1_{(0, \infty)}(x) \\ axe^{-x}x^{\alpha}1_{(0,\infty)}(x) & e^{-x}x^{\alpha}1_{(0,\infty)}(x)\end{pmatrix},$$
where $T(x) = \begin{pmatrix} 1 & ax \\ 0 & 1 \end{pmatrix}$, $\tilde{W}(x) = \begin{pmatrix} e^{-x^{2}} & 0 \\ 0 & e^{-x}x^{\alpha}1_{(0,\infty)}(x) \end{pmatrix}$, and $1_{(0,\infty)}(x)$ is the indicator function of the interval $(0,\infty)$, which takes the value $1$ for $x > 0$ and $0$ otherwise.
By Theorem \ref{main}, a sequence of orthogonal polynomials for $W$ is 
$$Q_{n}(x) = \begin{pmatrix} h_{n}(x) & a(\ell_{n+1}^{(\alpha)}(x) - h_{n}(x)x) \\ -an\frac{2^{n-1}}{\sqrt{\pi}}\Gamma(n+\alpha+1)h_{n-1}(x) & a^{2}n\frac{2^{n-1}}{\sqrt{\pi}}\Gamma(n+\alpha+1)h_{n-1}(x)x + \ell_{n}^{(\alpha)}(x)\end{pmatrix},$$
where $h_{n}(x)$ and $\ell_{n}^{(\alpha)}(x)$ are the $n$-th monic orthogonal polynomials of Hermite and Laguerre, respectively. Let $\delta = \partial^{2} + \partial (-2x)$ and $\delta_{\alpha} = \partial^{2} x + \partial (\alpha +1 -x)$ be the second-order differential operators of Hermite and Laguerre, respectively. It follows that $\Lambda_{n}(\delta-2) = \Lambda_{n+1}(2\delta_{\alpha})$.
Thus, by Theorem \ref{bisp}, the sequence $Q_{n}(x)$ is eigenfunction of the second-order differential operator $D = T(x) \operatorname{diag}(\delta - 2, \delta_{\alpha})T(x)^{-1}$,
$$D = \partial^{2} \begin{pmatrix} 1 & 2ax^{2}-ax \\ 0 & 2x   \end{pmatrix} + \partial \begin{pmatrix} -2x & 2ax(\alpha+3) \\ 0 & 2(\alpha+1-x)\end{pmatrix} + \begin{pmatrix} -2 & 2a(\alpha+1) \\ 0 & 0 \end{pmatrix},$$
with
$$Q_{n}(x) \cdot D = \begin{psmallmatrix} -2n - 2 & 0 \\ 0 & -2n\end{psmallmatrix} Q_{n}(x).$$
On the other hand, the sequence $Q_{n}(x)$ satisfies the following three-term recurrence relation 
$$Q_{n}(x)x = A_{n}Q_{n+1}(x) + B_{n}Q_{n}(x) + C_{n}Q_{n-1}(x),$$
with 
\begin{equation*}
    \begin{split}
        A_{n} & = \begin{psmallmatrix} 1 & \frac{an(2n+3+\alpha)}{2M_{n}a^{2}(n+1)(\alpha+n+1)+n} \\ 0 & \frac{n(M_{n}a^{2}+1)}{2M_{n}a^{2}(n+1)(\alpha+n+1)+n} \end{psmallmatrix}, \quad C_{n} = \begin{psmallmatrix} \frac{2M_{n}a^{2}(n+1)(\alpha+n+1)+1}{2M_{n}a^{2}+2} & 0 \\ \frac{M_{n}a(2n+\alpha+1)}{M_{n}a^{2}+1} & \alpha n + n^{2} \end{psmallmatrix}, \\
        B_{n} & = \begin{psmallmatrix} \frac{2M_{n}a^{2}(n+1)(2n+3+\alpha)(\alpha+n+1)}{2M_{n}a^{2}(n+1)(\alpha+n+1)+n} & \frac{(2\alpha n + 2n^{2} + 2 \alpha + 3n + 2)a}{2(M_{n}a^{2}+1)} \\ \frac{(2\alpha n + 2n^{2} + 2\alpha + 3n + 2) M_{n}a}{2M_{n}a^{2}(n+1)(\alpha+n+1)+n} & \frac{2n+\alpha+1}{M_{n}a^{2}+1}\end{psmallmatrix}.
    \end{split}
\end{equation*}
Where $M_{n} = \frac{n}{\sqrt{\pi}}2^{n-1}\Gamma(n+\alpha+1)$. Thus, the sequence $Q_{n}$ is bispectral.

\section{Irreducibility} \label{S5}
In this section, we study the irreducibility of the weight matrices introduced in \eqref{B} for $N = 2$, for some cases of $N = 3$, and for some particular cases of general size. 

\begin{prop} \label{irr 2}
    The $2 \times 2$ weight matrix $W$ defined in \eqref{W 2} supported on $(x_{0},x_{1})$ reduces to a direct sum of scalar weights if and only if $$w_{1}(x) = -a^{2}w_{2}(x)(b-x)(c-x)$$
    for some $b,\, c \in \mathbb{R}$ with $c>b$, and $(x_{0}, x_{1}) \subseteq (b, \, c)$.

    In particular, if $x_{0} = - \infty$, or $x_{1} = \infty$, then the weight matrix $W$ is irreducible.
\end{prop}
\begin{proof}
    Assume that $W$ reduces to a direct sum of scalar weights. Then there exists a non-scalar $W$-symmetric operator $F \in \mathcal{D}(W)$ of order zero (see \cite[Theorem 4.3]{TZ18}). That is, there exists a constant matrix $F$ in $\operatorname{Mat}_{2}(\mathbb{C})$ such that 
    \begin{equation}\label{eq}
        FW(x) = W(x)F^{\ast}, \quad x\in (x_{0},x_{1}),
    \end{equation}
    with $F \notin \mathbb{C}I$.
    We write $F = \begin{pmatrix} p_{1} && p_{2} \\ p_{3} && p_{4} \end{pmatrix}$. By comparing the entries in the equation \eqref{eq}, it follows that $p_{1}$, $p_{2}$, $p_{3}$ and $p_{4}$ are real numbers. If $p_{3} = 0$, then we have that $p_{2} = 0$ and $p_{1} = p_{4}$. Thus, $F = p_{1}I$, which is not possible. If $p_{3} \neq 0$, then the equation \eqref{eq} holds if and only if $w_{1}(x) = -a^{2}w_{2}(x)\left (x^{2} + ax \frac{(p_{1}-p_{4})}{p_{3}} + \frac{p_{2}}{p_{3}}\right )$. Note that if the polynomial $x^{2} + ax \frac{(p_{1}-p_{4})}{p_{3}} + \frac{p_{2}}{p_{3}}$ does not have two distinct real roots, then the scalar weight $w_{1}$ has non-positive values for all $x$, which cannot happen. Hence, $w_{1}(x) = -a^{2}w_{2}(x)(x-b)(x-c)$ with $b,\, c$ the distinct roots. If $b<c$, then $w_{1}(x)> 0$ almost everywhere on $(x_{0},x_{1})$ if and only if $(x_{0},x_{1}) \subseteq (b,c)$.

    Conversely, suppose that $w_{1}(x) = -a^{2}w_{2}(x)(x-b)(x-c)$ for some $b, \, c \in \mathbb{R}$ with $b\leq x_{0}<x_{1} \leq c$. We have that 
    $$W(x) = \begin{pmatrix} -{a}^{2}w_{2}(x)(bc-x(b+c))  
    && axw_{2}(x) \\ axw_{2}(x) && w_{2}(x) 
    \end{pmatrix}.$$
    For the constant nonsingular matrix $M = \begin{pmatrix} \frac{1}{a(b-c)} && -\frac{b}{b-c} \\ 1 && -ac \end{pmatrix}$ it follows that 
    $$MW(x)M^{\ast} = \begin{pmatrix} w_{2}(x)\frac{(x-b)}{(c-b)} && 0 \\ 0 && a^{2}w_{2}(x)(c-x)(c-b)\end{pmatrix}.$$
    Thus, the weight matrix $W$ reduces to a direct sum of scalar weights.
\end{proof}
In conclusion, for $N = 2$, the weight matrix $W$ is almost always irreducible. In particular, the $2 \times 2$ weight matrices discussed in the previous section are examples of irreducible weights. We now provide an example of a reducible weight constructed using scalar Jacobi weights.

\begin{example}
    Let $w_{1}(x) = a^{2}(1-x)^{\alpha+1}(1+x)^{\beta+1}$, and $w_{2}(x) = (1-x)^{\alpha}(1+x)^{\beta}$ be the scalar Jacobi weights supported on $(-1,1)$, where $\alpha, \, \beta > -1$, and $a\not=0$. The $2 \times 2$ weight matrix 
    $$W(x) = \begin{pmatrix} 1 && ax \\ 0 && 1 \end{pmatrix} \begin{pmatrix} w_{1}(x) && 0 \\ 0 && w_{2}(x) \end{pmatrix} \begin{pmatrix} 1 && 0 \\ ax && 1 \end{pmatrix} = (1-x)^{\alpha}(1-x)^{\beta}\begin{pmatrix} a^{2} && ax \\ ax && 1\end{pmatrix}$$
    reduces to 
    $$\begin{pmatrix}\frac{1}{2}(1-x)^{\alpha}(1+x)^{\beta+1} && 0 \\ 0 && 2a^{2}(1-x)^{\alpha+1}(1+x)^{\beta} \end{pmatrix}.$$
\end{example}

For the case of $N = 2$, we have a well-established criterion to determine when the weight matrix reduces. However, for $N = 3$, the situation is more complex. Although we can provide partial results, we have not yet developed a criterion to characterize when these weights are irreducible.

\begin{prop}
    Let $W$ be the $3 \times 3$ weight matrix defined in \eqref{W 3}. Assume that 
    $$
    w_{1}(x) \neq q_{2}(x)w_{2}(x), \quad w_{3}(x) \neq q_{1}(x)w_{2}(x), \quad \text{and} \quad w_{1}(x) \neq q_{0}(x)w_{3}(x),
    $$
    for any polynomials $q_{i}(x)$ of degree at most $i$. Then, the weight matrix $W$ is irreducible.
\end{prop}
\begin{proof}
    Let $F = \begin{pmatrix} p_{1} & p_{2} & p_{3} \\ p_{4} & p_{5} & p_{6} \\p_{7} & p_{8} & p_{9} \end{pmatrix} \in \mathcal{D}(W)$ be a $W$-symmetric operator of order zero. We have that $FW(x) = W(x)F^{\ast}$, then $W(x)^{-1}FW(x) = F^{\ast}$ is a constant matrix. The $(1,2)$-entry of $W^{-1}(x)FW(x)$ is given by
    $$-\frac{w_{2}(x)}{w_{1}(x)}((a_{1}^{2}p_{4}+a_{1}a_{2}p_{6})x^{2} +(a_{1}(p_{5}-p_{1})-a_{2}p_{3})x-p_{2}),$$
    which is equal to $\overline{p}_{4}$. Then, we have that $(a_{1}^{2}p_{4}+a_{1}a_{2}p_{6})x^{2} +(a_{1}(p_{5}-p_{1})-a_{2}p_{3})x-p_{2}=0$ and $p_{4} = 0$. From here, we obtain that 
    $p_{2} = 0, p_{6} = 0$, $p_{5} = p_{1} + \frac{a_{2}p_{3}}{a_{1}}$. Now, we have that the $(3,2)$-entry of $W(x)^{-1}FW(x)$ is
    $$\frac{w_{2}(x)}{w_{3}(x)}( (a_{1}p_{7}-a_{2}p_{1}+a_{2}p_{9}-\frac{a_{2}^{2}}{a_{1}}p_{3} )x+p_{8}).$$
    Thus, we have that $(a_{1}p_{7}-a_{2}p_{1}+a_{2}p_{9}-\frac{a_{2}^{2}}{a_{1}}p_{3} )x+p_{8} = 0$. Then, $p_{7} = \frac{a_{2}}{a_{1}^{2}}(a_{1}(p_{1}-p_{9})+a_{2}p_{3})$, and $p_{8} = 0$. Then, the $(1,3)$-entry is $\frac{w_{3}}{w_{1}}p_{3}$, which implies that $p_{3} = 0$. Finally, from $(2,3)$-entry we obtain that $p_{1} = p_{9}$. Thus, we obtain that $F = p_{1} I$. Hence, there are no nonscalar $W$-symmetric operators of order zero in $\mathcal{D}(W)$. Thus, the statement holds.
\end{proof}

\begin{prop}
     Let $W$ be the $3\times 3$ weight matrix as defined in \eqref{W 3}. If $w_{1} = w_{3}$, then the weight $W$ reduces. 
\end{prop}
\begin{proof}
       We take the constant matrix $M = \begin{pmatrix}1 & 0 & -\frac{a_{1}}{a_{2}} \\ 0 & 1 & 0 \\ \frac{a_{1}a_{2}}{a_{1}^{2} + a_{2}^{2}} & 0 & \frac{a_{2}^{2}}{a_{1}^{2}+a_{2}^{2}} \end{pmatrix}$. It follows that 
    $$MW(x)M^{\ast} = \begin{pmatrix} \frac{a_{1}^{2}+a_{2}^{2}}{a_{2}^{2}}w_{1}(x) & 0 & 0 \\ 0 & w_{2}(x) & a_{2}xw_{2}(x) \\ 0 & a_{2}xw_{2}(x) & a_{2}^{2}x^{2}w_{2}(x) + \frac{a_{2}^{2}}{a_{1}^{2}+a_{2}^{2}}w_{2}(x)\end{pmatrix}.$$
    Thus, $W$ reduces to a direct sum $W_{1} \oplus W_{2}$ with $W_{1}(x) = \frac{a_{1}^{2}+a_{2}^{2}}{a_{2}^{2}}w_{1}(x)$, and $W_{2}$ the $2\times 2$ weight $W_{2}(x) = \begin{pmatrix}w_{2}(x) & a_{2}xw_{2}(x) \\ a_{2}xw_{2}(x) & a_{2}^{2}x^{2}w_{2}(x) + \frac{a_{2}^{2}}{a_{1}^{2}+a_{2}^{2}}w_{2}(x) \end{pmatrix}$.
\end{proof}

\begin{remark}
    For the case $N = 3$, if either $w_{1} = w_{2}$ or $w_{2} = w_{3}$, the weight $W$ defined in \eqref{W 3} does not necessarily reduce. For instance, consider $w_{1}(x) = e^{-x^{2}+2bx}$ for some $b \in \mathbb{R}-\{0\}$, and $w_{2}(x) = w_{3}(x) = e^{-x^{2}}$. In this case, we have
    $$W(x) = e^{-x^{2}}\begin{pmatrix} a_{1}^{2}x^{2}+e^{2bx} & a_{1}x & a_{1}a_{2}x^{2} \\ a_{1}x & 1 & a_{2}x \\ a_{1}a_{2}x^{2} & a_{2}x & a_{2}x^{2}+  1 \end{pmatrix}.$$
    $W$ does not reduce because its algebra $D(W)$ only contains scalar operators of order zero.
\end{remark}

As the size $N$ increases, characterizing all situations in which the weight $W$ is irreducible becomes more complex. In the next proposition, for size $N \times N$, we will provide a sufficient condition to ensure the irreducibility of the weight matrix. This result was proven in \cite{BP24-1} by studying the right Fourier algebras.

\begin{prop}[\cite{BP24-1}, Prop.~3.7]
    Let $w_{1},w_{2},\ldots,w_{N}$ be scalar weights supported on the same interval such that $w_{i}(x)w_{j}(x)^{-1}$ is not a rational function for all $i \not = j$. Then, the weight matrix 
    $$W(x) = T(x) \operatorname{diag}(w_{1}(x),\ldots,w_{N}(x))T(x)^{\ast}$$ as defined in \eqref{B} is an irreducible weight. 
\end{prop}

\ 

The above proposition provides a sufficient condition for generating many irreducible examples of weight matrices for arbitrary sizes. However, there are several other cases where we obtain irreducible weights. For instance, the Laguerre-type weight in \eqref{ej Lag 5} is irreducible because its algebra, $\mathcal{D}(W)$, contains only scalar operators of order zero. Another case where the weight does not reduce occurs when we construct the weight $W$ from a collection of scalar weights given by $w_{1}, \ldots, w_{N}$, where $w_{i+1}(x) = w_{1}(x)x^{i-1}$. We have proven that for $N \leq 10$, the algebra $\mathcal{D}(W)$ contains only scalar operators of order zero. We conjecture that this result holds for arbitrary $N$. 

For example, the $10 \times 10$ Laguerre-type weight
$$W(x) = T(x) \operatorname{diag}(e^{-x}x^{\alpha}, e^{-x}x^{\alpha+1}, \ldots, e^{-x}x^{\alpha+9}) T(x)^{\ast}$$
is an irreducible weight.

\section{Appendix}

\begin{thm}\label{det tec}
    Let $a_{1},\ldots,a_{N-1},b_{1},\ldots,b_{N-1} \in \mathbb{R}$. We consider the $N \times N$ tridiagonal matrix 
    $$K = \begin{psmallmatrix} 1 & a_{1} & 0 & 0 & 0 & \ldots \\
    b_{1} & 1 & b_{2} & 0 & 0 & \ldots \\
    0 & a_{2} & 1 & a_{3} & 0 & \ldots \\
    0 & 0 & b_{3} & 1 & b_{4} & \ldots \\
    \vdots & \vdots & \vdots & \ddots & \ddots & \ddots\end{psmallmatrix}.$$
    For $i = 1, \ldots, N-1$, we define the numbers $\rho_{i} = -a_{i}b_{i}$. Then, 
    $$\det(K) = 1+\sum_{k=1}^{[(N+1)/2]} \sum_{1\leq i_{1}<i_{2}<\ldots<i_{k}\leq N-1}\rho_{i_{1}}\cdots \rho_{i_{k}}.$$
\end{thm}
\begin{proof}
    For $N = 2$, we have $K = \begin{psmallmatrix} 1 & a_{1} \\b_{1} & 1 \end{psmallmatrix}$, and $\rho_{1} = -a_{1}b_{1}$. We have $\det(K) = 1 - a_{1}b_{1} = 1 + \rho_{1}$. Thus, the statement holds. 
    We proceed by induction on $N$. We have $K = \begin{psmallmatrix} 1 & a_{1} & 0 & 0 & 0 & \ldots \\
    b_{1} & 1 & b_{2} & 0 & 0 & \ldots \\
    0 & a_{2} & 1 & a_{3} & 0 & \ldots \\
    0 & 0 & b_{3} & 1 & b_{4} & \ldots \\
    \vdots & \vdots & \vdots & \ddots & \ddots & \ddots\end{psmallmatrix}_{N\times N}$. We expand the determinant along the first row and obtain 
    $$\det(K) = 1\cdot \det \begin{psmallmatrix} 1 & a_{2} & 0 & 0 & 0 & \ldots \\
    b_{2} & 1 & b_{3} & 0 & 0 & \ldots \\
    0 & a_{3} & 1 & a_{4} & 0 & \ldots \\
    0 & 0 & b_{4} & 1 & b_{5} & \ldots \\
    \vdots & \vdots & \vdots & \ddots & \ddots & \ddots\end{psmallmatrix} - a_{1} \det \begin{psmallmatrix} b_{1} & b_{2} & 0 & 0 & 0 & \ldots \\
     0 & 1 & a_{3} & 0 & 0 & \ldots \\
     0 & b_{3} & 1 & b_{4} & 0 & \ldots \\
     0 & 0 & a_{4} & 1 & a_{5} & \ldots \\
     \vdots & \vdots & \vdots & \ddots & \ddots & \ddots\end{psmallmatrix}.$$
     Expanding the second determinant on the right-hand side along the first column, and using the inductive hypothesis, we obtain
     $$\det(K) = 1 + \sum_{k=1}^{[N/2]} \sum_{2 \leq i_{1} < \ldots < i_{k}\leq N-1} \rho_{i_{1}}\cdots \rho_{i_{k}} + \rho_{1}\left ( 1 + \sum_{k=1}^{[(N-1)/2]} \sum_{3 \leq i_{1} < \ldots < i_{k}\leq N-1} \rho_{i_{1}}\cdots \rho_{i_{k}}\right ).$$
     Thus, the statement holds.
\end{proof}

\bibliographystyle{plain}
\bibliography{referencias}

\end{document}